\documentclass[a4paper,10pt]{amsart}
\usepackage{amsmath, 
 amscd, 
 amssymb, 
 mathrsfs, 
 mathtools, 
 enumerate, 
 fullpage}
\usepackage[pagebackref,breaklinks=true]{hyperref}

\DeclareMathOperator{\Hom}{Hom}

\newcommand{\cE}{\mathcal{E}}
\newcommand{\cF}{\mathcal{F}}
\newcommand{\cO}{\mathcal{O}}
\newcommand{\cW}{\mathcal{W}}

\newcommand{\CC}{\mathbf{C}}
\newcommand{\QQ}{\mathbf{Q}}
\newcommand{\ZZ}{\mathbf{Z}}
\newcommand{\bfk}{\mathbf{k}}

\newcommand{\Qp}{{\QQ_p}}
\newcommand{\QQbar}{\overline{\QQ}}
\newcommand{\Qpbar}{\overline{\QQ}_p}
\newcommand{\Zp}{\ZZ_p}
\newcommand{\tU}{\widetilde{U}}

\newcommand{\into}{\hookrightarrow}
\renewcommand{\ge}{\geqslant}
\renewcommand{\le}{\leqslant}
\renewcommand{\ss}{\spadesuit}
\newcommand{\ds}{\diamondsuit}

\numberwithin{equation}{subsection}

\theoremstyle{plain}
\newtheorem{theorem}{Theorem}[section]
\newtheorem{lemma}[theorem]{Lemma}
\newtheorem{proposition}[theorem]{Proposition}
\newtheorem{corollary}[theorem]{Corollary}

\theoremstyle{definition}
\newtheorem{definition}[theorem]{Definition}

\theoremstyle{remark}
\newtheorem{remark}[theorem]{Remark}

\title{A note on $p$-adic Rankin--Selberg $L$-functions}

\author{David Loeffler}
\address{Mathematics Institute, Zeeman Building, University of Warwick, Coventry CV4 7AL, UK}
\urladdr{\href{http://orcid.org/0000-0001-9069-1877}{0000-0001-9069-1877}}
\email{d.a.loeffler@warwick.ac.uk}

\thanks{The author's research was supported by a Royal Society University Research Fellowship.}


\begin{document}
 
 \begin{abstract}
  We prove an interpolation formula for the values of certain $p$-adic Rankin--Selberg $L$-functions associated to non-ordinary modular forms.
 \end{abstract} 
 
 \maketitle
  
 \section{Introduction}
 
  \subsection{Background} Let $f_1$, $f_2$ be two modular eigenforms, of weights $k_1 > k_2$. Then there is an associated Rankin--Selberg $L$-function $L(f_1, f_2, s)$, which is defined by a Dirichlet series $\sum c_n n^{-s}$ such that for $\ell$ prime we have $c_\ell = a_\ell(f) a_\ell(g)$.
  
  If $p$ is prime, and $f_1$ is \emph{ordinary} at $p$, then a well-known construction due to Panchishkin \cite{panchishkin82} and (independently) Hida \cite{hida85} gives rise to a $p$-adic Rankin--Selberg $L$-function $L_p(f_1, f_2, \sigma)$. This is a $p$-adic analytic function on the space $\cW$ of continuous characters of $\Zp^\times$, with the property that if $\sigma$ is a locally algebraic character $z \mapsto z^j \chi(z)$, with $j$ in the critical range $k_2 \le j \le k_1 - 1$ and $\chi$ of finite order, then
  \[ L_p(f_1, f_2, \sigma) = (\star) \cdot L(f_1, f_2, \chi^{-1}, j) \]
  where $(\star)$ is an explicit factor. Hida subsequently showed in \cite{hida88} that if $f_2$ is also ordinary, then $L_p(f_1, f_2, \sigma)$ extends to a 3-variable analytic function in which the forms $f_1$ and $f_2$ are allowed to vary in Hida families $\mathcal{F}_1, \mathcal{F}_2$. The existence of this $p$-adic $L$-function plays a major role in several recent works on arithmetic of Rankin--Selberg $L$-functions, in particular appearing in the explicit reciprocity law for the Euler system of Beilinson--Flach elements \cite{BDR-BeilinsonFlach, BDR-BeilinsonFlach2, KLZ17} (which is in turn crucial for several other recent works such as \cite{buyukboduklei16, castella-heights-BF, Dasgupta-factorization}).
  
  It is natural to seek a generalisation of this construction to non-ordinary eigenforms, and variation in Coleman families. For fixed $f_1$ and $f_2$ of level prime to $p$ and satisfying a suitable ``small slope'' hypothesis, such a construction was carried out by My \cite{my91}, but allowing variation in families has proved to be substantially more difficult. A construction of a 3-variable $p$-adic $L$-function with the expected interpolating property was initially announced in \cite{Urban-nearly-overconvergent}, but an error in this construction was subsequently found, and (to the best of the this author's knowledge) this has not been fully resolved at the present time\footnote{See note on next page.}.
  
  In the author's recent work with Zerbes \cite[Theorem 9.3.2]{loefflerzerbes16}, it was shown that there exists a 3-variable $p$-adic $L$-function with the expected interpolating property at \emph{crystalline} points (i.e. where $f_1$ and $f_2$ are $p$-stabilisations of eigenforms of level prime to $p$, and $\chi$ is trivial). Moreover, this $p$-adic $L$-function is related by an explicit reciprocity law to the Euler system of Beilinson--Flach elements, as in the ordinary case. Unfortunately, we were not able to establish unconditionally that the $p$-adic $L$-function thus constructed also had the expected interpolation property at non-crystalline points, so our results fell short of giving a full proof of the results announced in \cite{Urban-nearly-overconvergent}.
  
  This gap in the published literature has become increasingly troublesome, since several papers have now been published which assume this stronger interpolation property; these include several papers making major contributions to famous open problems, such as the Iwasawa main conjecture for supersingular elliptic curves \cite{buyukboduklei16b, wan15} and the Birch--Swinnerton-Dyer conjecture in analytic rank 1 \cite{jetchevskinnerwan}.
  
  \subsection{Aims of this paper} The purpose of this note is to give a proof of an interpolation formula for the $L$-function of \cite{loefflerzerbes16} at all critical points, crystalline or otherwise, in a certain special case. The assumption we make is that the Coleman family $\cF_2$ is ordinary, although $\cF_1$ may not be; this suffices for the applications in the papers cited above (all of which correspond to the case where $\cF_2$ is an ordinary family of CM-type). The present author is cautiously optimistic that it might be possible to push these methods further in order to give a full proof of the results announced in \cite{Urban-nearly-overconvergent}, but believes it is in the interests of the research community to release this partial proof without further delay, in order to place the already-published papers conditional on this result on a firm footing.
  
  Our strategy will be to relate the 3-variable ``geometric'' $p$-adic $L$-function, constructed using Beilinson--Flach elements, with two families of ``analytic'' $p$-adic $L$-functions. These 2-variable functions, denoted here by superscripts $\spadesuit$ and $\diamondsuit$, are defined over 2-variable slices of the full 3-variable parameter space. Their construction involves nearly-overconvergent forms of a fixed degree, and therefore can be carried out using the methods of \cite{Urban-nearly-overconvergent} without the technical issues which arise when the degree of near-overconvergence is allowed to vary. The assumption that the second Coleman family $\cF_2$ is ordinary implies that it is defined over an entire component of weight space; this gives sufficient ``room'' to move along $\spadesuit$ and $\diamondsuit$ families from an arbitrary critical point to a crystalline one at which the results of \cite{KLZ17} can be applied.
  
  A secondary aim of this paper is to make the interpolation formula for the resulting $p$-adic $L$-function completely explicit, at least in the most important cases. This calculation is not new, but a precise statement of the formula seems to be difficult to find in the existing references (particularly in the non-crystalline cases); so we have given careful statements in Propositions \ref{prop:interp-formula} and \ref{prop:interp-formula2}, and an outline sketch of their proofs in an appendix.
  
  \subsubsection*{Note added during review}
  
   Since the initial version of this paper was released, the author has learned of the article \cite{AIU} in preparation, which circumvents the problems with \cite{Urban-nearly-overconvergent} via a new approach to nearly-overconvergent modular forms (as sections of a certain sheaf of Banach modules). This should in due course lead to a proof of an analogue of Theorem 6.3 of the present paper for arbitrary pairs of Coleman families, without the restriction imposed here that $\cF_2$ be ordinary. However, the author believes that there is still value in making this note available, since the preprint \cite{AIU} has not yet been published, and the preliminary version of \cite{AIU} seen by the author only considers families over the ``centre'' of weight space and thus does not cover most non-crystalline classical points.
  
  \subsection*{Acknowledgements} I am grateful to Eric Urban and Xin Wan for helpful comments on the topic of this paper, and to Xin Wan in particular for encouraging me to write it up. Part of the work described in the paper was carried out during a visit to the Institute for Advanced Study in Princeton in the spring of 2016, and I am very grateful to the IAS for their hospitality.
    
 \section{Complex Rankin--Selberg $L$-functions and period integrals}
 
  \subsection{The complex $L$-function}
  
   Let $k, k'$ be positive integers, and $f_1$, $f_2$ two new, normalised cuspidal modular eigenforms of weights $k_1, k_2$ (and some levels $N_1, N_2$). We assume $k_1 \ge k_2$ without loss of generality. 
   
   \begin{definition}
    The (imprimitive) \emph{Rankin--Selberg $L$-function} of $f_1$ and $f_2$ is the Dirichlet series
    \[ 
     L^{\mathrm{imp}}(f_1, f_2, s) 
     = L_{(N_1 N_2)}(\varepsilon_1 \varepsilon_2, 2s +2 - k_1 - k_2) \cdot \sum_{n \ge 1} a_n(f_1) a_n(f_2) n^{-s}. 
    \]
    More generally, if $\chi$ is a Dirichlet character of conductor $N_\chi$ we set
    \[ 
     L^{\mathrm{imp}}(f_1, f_2, \chi, s) = L_{(N_1 N_2 N_\chi)}(\varepsilon_1 \varepsilon_2 \chi^2, 2s + 2 - k_1 - k_2) \cdot \sum_{\substack{n \ge 1 \\ (n, N_\chi) = 1}} a_n(f_1) a_n(f_2) \chi(n) n^{-s}.
    \]
   \end{definition}

   This $L$-function has an Euler product, in which the local factor for a primes $\ell \nmid N_1 N_2 N_\chi$ is given by $P_\ell(f_1, f_2, \chi(\ell) \ell^{-s})^{-1}$, where
   \[ 
    P_\ell(f_1, f_2, X) = (1 - \alpha_1 \alpha_2  X)(1 - \alpha_1 \beta_2 X) (1 - \beta_1 \alpha_2  X)(1-\beta_1 \beta_2  X).
   \]
   Here $\alpha_1, \beta_1$ denote the roots of the polynomial $X^2 - a_\ell(f_1) X + \ell^{k-1} \varepsilon_1(\ell)$, and similarly for $\alpha_2, \beta_2$.
  
   \begin{remark}
    We refer to this $L$-function as an ``imprimitive'' $L$-function since it differs by finitely many Euler factors from the $L$-function of the motive associated to $f_1 \otimes f_2 \otimes \chi$ (the ``primitive'' Rankin--Selberg $L$-function). The only primes $\ell$ at which the local Euler factors can differ are those $\ell$ dividing at least two of the three integers $N_1, N_2, N_\chi$; so if these are pairwise coprime, then the primitive and imprimitive $L$-functions coincide.
   \end{remark}
   
   It is well known that $L^{\mathrm{imp}}(f_1, f_2, \chi, s)$ has meromorphic continuation to all $s \in \CC$. It is entire unless $k_1 = k_2$ and $f_2 = f_1 \otimes \varepsilon_1^{-1} \chi^{-1}$, in which case there is a simple pole at $s = k_1$. The critical values are those in the interval $k_2 \le s \le k_1 - 1$.

  \subsection{A Petersson product formula}
  
   Now let $p$ be prime; and choose an embedding $\QQbar \into \Qpbar$.
   
  \begin{definition}
   A \emph{locally algebraic character} of $\Zp^\times$ is a homomorphism $\Zp^\times \to \Qpbar^\times$ of the form $x \mapsto x^n \chi(x)$, where $n \in \ZZ$ and $\chi$ is a finite-order character (equivalently, a Dirichlet character of $p$-power conductor). We denote this character by ``$n + \chi$''.
  \end{definition}
  
  \begin{definition}
   By a \emph{$p$-stabilised newform} of tame level $N$, where $N$ is an integer coprime to $p$, we shall mean a normalised cuspidal Hecke eigenform of level $\Gamma_1(Np^r)$, for some $r \ge 1$, such that either $f$ is a newform, or $f$ is a $U_p$-eigenform in the two-dimensional space of oldforms associated to some newform of level $N$. In the latter case, we say $f$ is \emph{crystalline}.
   
   We define the \emph{weight-character} of $f$ to be the locally-algebraic character $\kappa$ of $\Zp^\times$ defined by $\kappa = k + \varepsilon_p$, where $k$ is the weight of $f$ and $\varepsilon_p$ is the $p$-part of the Nebentypus character of $f$.
   
   If $f$ is a $p$-stabilised newform, we denote by $f^c$ the unique $p$-stabilised newform with the same weight-character as $f$ satisfying
   \[ a_n(f^c) = \varepsilon_{N, f}(n)^{-1} a_n(f), \]
   where $\varepsilon_{N, f}$ is the prime-to-$p$ part of the Nebentypus of $f$, for all $(n, N) = 1$ (even if $p \mid n$).
  \end{definition}
  
  \begin{remark}
   Note that if $f$ is a $p$-stabilised newform whose nebentypus is trivial at $p$, then $f^c$ has the same Hecke eigenvalues away from $p$ as the conjugate form $f^*$ defined by $f^*(\tau) = \overline{f(-\bar \tau)}$. However, $f^c$ and $f^*$ do not generally have the same $U_p$-eigenvalue; in particular $f^c$ is ordinary if $f$ is (which is not true of $f^*$). On the other hand, if $f$ has non-trivial character at $p$, then the Hecke eigenvalues of $f^c$ and $f^*$ away from $p$ are different.
  \end{remark}
  
  Let $f_1, f_2$ be $p$-stabilised newforms of some tame levels $N_1, N_2$, and let $\kappa_1 = k_1 + \varepsilon_{1, p}, \kappa_2 = k_2 + \varepsilon_{2, p}$ be their weight-characters. We choose an integer $N$ divisible by both $N_1$ and $N_2$, and with the same prime factors as $N_1 N_2$. Given $\sigma = j + \chi$ a locally algebraic character, we consider the formal power series
  \[ 
   \mathcal{E}_N(\kappa_1, \kappa_2, \sigma) \coloneqq \sum_{\substack{n \ge 1 \\ p \nmid n}} \left( \sum_{d \mid n} d^{\sigma - \kappa_2} \left(\tfrac{n}{d}\right)^{\kappa_1 - \sigma - 1} \left[ e^{2\pi i d / N} + (-1)^{\kappa_1 - \kappa_2} e^{-2\pi i d / N}\right] \right) q^n.
  \]
  
  \begin{lemma}
   If $1 \le k_2 \le j \le k_1-1$, then $\mathcal{E}_N(\kappa_1, \kappa_2, \sigma)$ is the $q$-expansion of a nearly-holomorphic modular form of weight $k_1 - k_2$, level dividing $Np^\infty$, and degree at most $\min(k_1 - 1 - j, j - k_2)$, on which the diamond operators at $p$ act via the character $\varepsilon_{1, p} - \varepsilon_{2, p}$.
  \end{lemma}
  
  \begin{proof}
   See \cite[\S 5.3]{leiloefflerzerbes14}.
  \end{proof}
  
  If $\Pi^{\mathrm{hol}}$ denotes Shimura's holomorphic projector, then the cuspidal modular form
  \[ \Pi^{\mathrm{hol}}\left(f_2 \cdot \mathcal{E}_N(\kappa_1, \kappa_2, \sigma)\right) \]
  has level dividing $Np^\infty$, and its weight-character agrees with that of $f_1$ (and thus also of $f_1^c$).
  
  \begin{definition}
   Suppose $f_1$ has finite slope (that is, $a_p(f) \ne 0$). We let $\lambda_{f_1^c}$ denote the unique linear functional on $S_{k_1}(N_1 p^\infty, \bar{\varepsilon}_{1, p})$ which factors through the Hecke eigenspace associated to $f_1^c$, and maps the normalised eigenform $f_1^c$ itself to 1. We extend this to forms of tame level $N$ by composing with the trace map.
  \end{definition}
  
  \begin{definition}
   We set
   \[ I(f_1, f_2, \sigma) = N^{\kappa_1 + \kappa_2 - 2\sigma - 2} \cdot \lambda_{f_1^c}\Big( \Pi^{\mathrm{hol}}\left(f_2 \cdot \mathcal{E}_N(\kappa_1, \kappa_2, \sigma)\right) \Big).\]
  \end{definition}
  
  \begin{theorem}[Rankin--Selberg, Shimura]
   If $1 \le k_2 \le j \le k_1-1$ then we have
   \[ I(f_1, f_2, j + \chi) = (\star) \cdot L^{\mathrm{imp}}(f_1, f_2, \chi^{-1}, j) \]
   where $(\star)$ is an explicitly computable factor.
  \end{theorem}
  
  We shall not give the precise form of the factor $(\star)$ in all possible cases, since this rapidly becomes messy, but we shall give a selection of useful cases. First, we treat the case where $f_1$ and $f_2$ are crystalline, hence $p$-stabilisations of forms $f_1^\circ, f_2^\circ$ of levels $N_1, N_2$ coprime to $p$. We write $\alpha_i$ for the $U_p$-eigenvalue of $f_i$, so that $\alpha_i$ is a root of the Hecke polynomial of $f_i^\circ$ at $p$, and $\beta_i$ for the other root of this polynomial. We assume\footnote{This assumption is known to be true if $k_1 = 2$, and is known to follow from the Tate conjecture if $k_1 \ge 3$ \cite{colemanedixhoven98}.} that $\alpha_1 \ne \beta_1$.  
  
  We define certain local Euler factors at $p$, as in \cite{BDR-BeilinsonFlach} and \cite[Theorem 2.7.4]{KLZ17}, by
  \begin{gather*}
   \cE(f_1) = \left( 1 - \frac{\beta_1}{p \alpha_1}\right), \qquad \cE^*(f_1) = \left( 1 - \frac{\beta_1}{\alpha_1}\right),\\
   \cE(f_1, f_2, j + \chi)\! =\!
   \begin{cases}
    \left( 1 - \frac{p^{j-1}}{\alpha_1 \alpha_2}\right) \left( 1 - \frac{p^{j-1}}{\alpha_1 \beta_2}\right) \left( 1 - \frac{\beta_1 \alpha_2}{p^j}\right) \left( 1 - \frac{\beta_1 \beta_2}{p^j}\right) &\text{if $\chi=1$,}\\[2mm]
    G(\chi)^2 \cdot \left( \frac{ p^{2s-2} }{\alpha_1^2 \alpha_2 \beta_2}\right)^r &\text{ if $\chi$ has conductor $p^r > 1$.}
   \end{cases}
  \end{gather*}
  Here $G(\chi)$ is the Gauss sum $\sum_{a \in (\ZZ/p^r\ZZ)^\times} \chi(a) e^{2\pi i a / p^r}$.
  
  \begin{proposition}
   \label{prop:interp-formula}
   In the above setting, we have
   \[ 
    I(f_1, f_2, j+\chi) = \frac{\cE(f_1, f_2, j)}{\cE(f_1) \cE^*(f_1)}
    \cdot
    \frac{(j-1)!(j-k_2)! i^{k_1 - k_2}}{\pi^{2j + 1 - k_2}\, 2^{2j + k_1 - k_2}\, \langle f_1^\circ, f_1^\circ\rangle_{N_1} } L^{\mathrm{imp}}(f_1^\circ, f_2^\circ, \chi^{-1}, j).
   \]
  \end{proposition}
     
  \begin{remark}
   For $\chi$ trivial, this formula is standard, and its derivation can be found in many references  such as \cite{BDR-BeilinsonFlach, leiloefflerzerbes14, loefflerzerbes16}. For $\chi$ non-trivial, references are more scant; many sources, such as \cite{hida88}, give more general but less explicit formulas, and the work involved in recovering a completely explicit form for all the local factors is routine but unpleasant. For the convenience of the reader we give an account of the main steps required to evaluate $I(f_1, f_2, j+\chi)$ in this case in an appendix to this paper.
  \end{remark}
  
  The other case we shall consider is that where $f_1$ is still assumed crystalline, but $f_2$ has some non-trivial character $\varepsilon_{2, p}$ at $p$, and neither $\chi$ nor $\chi' = \chi \varepsilon_{2, p}^{-1}$ is trivial. We define $\beta_2 = p^{k_2 - 1} \varepsilon_{2, N}(p) / \alpha_2$, and we let the conductor of $\chi$ (resp.~$\chi'$) be $p^r$ (resp.~$p^{r'}$).
  
  \begin{proposition}
   \label{prop:interp-formula2}
   In this setting we have
   \begin{multline*} 
    I(f_1, f_2, j+\chi) = \left(\tfrac{p^{j-1}}{\alpha_1 \alpha_2} \right)^r G(\chi) \left(\tfrac{p^{j-1}}{\alpha_1 \beta_2} \right)^{r'} G(\chi') \\
    \times
    \frac{(j-1)!(j-k_2)! i^{k_1 - k_2}}{\cE(f_1) \cE^*(f_1)\pi^{2j + 1 - k_2}\, 2^{2j + k_1 - k_2}\, \langle f_1^\circ, f_1^\circ\rangle_{N_1} } L^{\mathrm{imp}}(f_1^\circ, f_2, \chi^{-1}, j).
   \end{multline*}
  \end{proposition}
  
 \section{Overconvergent families}
  
  Let us fix a finite extension $L / \Qp$ (contained in our fixed choice of algebraic closure $\Qpbar$).
  
  \begin{definition}
   Let the \emph{weight space}, $\cW$, be the rigid-analytic space over $L$ parametrising continuous characters of $\Zp^\times$, so that for an affinoid $L$-algebra $A$, we have $\cW(A) = \Hom(\Zp^\times, A^\times)$. 
  \end{definition}
  
  As in \cite{KLZ17}, we identify both $\ZZ$ and the set of Dirichlet characters of $p$-power order with subsets of $\cW(\bar{L})$ in the natural fashion; and we denote the group law on $\cW$ additively. If $\kappa = k + \chi$ is a locally algebraic character, we write $w(\kappa) \coloneqq k$.
  
  Now let $N$ be an integer coprime to $p$. It will be convenient to assume that $L$ contains the $N$-th roots of unity; let $\zeta_N \in L^\times$ denote the image of $e^{2\pi i / N} \in \overline{\QQ}$ under our chosen embedding.
     
  \begin{lemma}
   The power series in $E^{[p]}_\bfk$ and $F^{[p]}_\bfk$ in $\cO(\cW)[[q]]$ given by
   \[ E^{[p]}_\bfk \coloneqq \sum_{\substack{n \ge 1 \\ p \nmid n}} \left( \sum_{d \mid n} d^{\bfk - 1} (\zeta_N^d + (-1)^\bfk \zeta_N^{-d})\right) q^n \]
   and
   \[ F^{[p]}_\bfk \coloneqq \sum_{\substack{n \ge 1 \\ p \nmid n}} \left( \sum_{d \mid n} \left(\frac{n}{d}\right)^{\bfk - 1} (\zeta_N^d + (-1)^\bfk \zeta_N^{-d})\right) q^n \]
   are both the $q$-expansions of families of overconvergent modular forms over $\cW$ of tame level $\Gamma_1(N)$ and weight $\bfk$ (with radius of overconvergence bounded below over any affinoid in $\cW$).\qed
  \end{lemma}
  
  \begin{lemma}
   Let $\chi$ be a Dirichlet character of $p$-power conductor, with values in $L$. Then, for any family of overconvergent modular forms $\cF$ of tame level $\Gamma_1(N)$ and weight $\kappa: \Zp^\times \to A^\times$, where $A$ is an affinoid algebra, the power series defined by
   \[ \theta^{\chi} \cF \coloneqq \sum_{\substack{n \ge 1 \\ p \nmid n}} a_n(\cF) \chi(n) q^n \]
   is the $q$-expansion of a family of overconvergent forms over $A$, of weight $\kappa + 2\chi$.
  \end{lemma}
 
  \begin{proof}[Sketch of proof]
   Let $\chi$ have conductor $p^r$. Then there is a ``twisting homomorphism'' $t_j: X_1(Np^{2r}) \to X_1(N)$, given in terms of complex uniformizations by $\tau \mapsto \tau + \tfrac{j}{p^r}$, for any $j \in \ZZ / p^r \ZZ$. This preserves the component of the ordinary locus containing $\infty$, and extends to all sufficiently small overconvergent neighbourhoods of it, so it induces a pullback map on overconvergent modular (or cusp) forms. Since $\theta^{\chi}\mathcal{F}$ is equal to $\sum_{j \in (\ZZ / p^r \ZZ)^\times} \chi(j)^{-1} t_j^*(\mathcal{F})$ up to a constant, it is overconvergent of level $\Gamma_1(Np^{2r})$ and weight-character $\kappa$; and the diamond operators at $p$ act on it via $\chi^2$, so it descends to an overconvergent form of level $\Gamma_1(N) \cap \Gamma_0(p^{2r})$ and weight $\kappa + 2\chi$. Via the canonical-subgroup map we can regard it as an overconvergent form of level $N$.
  \end{proof} 
  
  In order to allow more general twists, we work with families of nearly-overconvergent modular forms (of some finite degree $r \ge 0$), in the sense of \cite[\S 3.3.2]{Urban-nearly-overconvergent}. If $\tau$ is a locally algebraic weight with $w(\tau) \ge 0$, we may thus define $\theta^{\tau}(\cF)$ as a family of nearly-overconvergent forms of weight $\kappa + 2\tau$ and degree $w(\tau)$.
     
  \begin{lemma}
   If $\cF$ is a Coleman family (a family of overconvergent normalised eigenforms of finite slope), new of some tame level $N$, defined over some affinoid $A \to \cW$, then there is a unique tame level $N$ Coleman family $\cF^c$ over $A$ satisfying
   \[ a_n(\cF^c) = \varepsilon_N(n)^{-1} a_n(\cF) \]
   for all $(n, N) = 1$ (including $n = p$). Here $\varepsilon_N: (\ZZ / N\ZZ)^\times \to L^\times$ is the prime-to-$p$ nebentype of $\cF$.
  \end{lemma}
 
  \begin{proof}
   This is proved in the same way as the previous lemma.
  \end{proof}
  
  We now recall the construction of the universal object parametrising Coleman families -- the eigencurve:
  
  \begin{definition}
   Let $\mathcal{C}_N$ denote the Coleman--Mazur--Buzzard cuspidal eigencurve, of tame level $N$.
  \end{definition}
  
  By definition, $\mathcal{C}_N$ is a reduced rigid space, equidimensional of dimension 1, equipped with a morphism $\mathcal{C}_N \to \cW$; and there is a universal eigenform over $\mathcal{C}_N$ -- that is, $\mathcal{C}_N$ comes equipped with a power series $\mathcal{F}^{\mathrm{univ}} = \sum a_n q^n \in \cO(\mathcal{C}_N)[[q]]$, with $a_1 = 1$ and $a_p$ invertible on $\mathcal{C}_N$, with the following universal property:
  \begin{quotation}
   For any affinoid $X$ with a weight morphism $\kappa: X \to \cW$, and any family of finite-slope eigenforms $\cF_X$ over $X$ of tame level $N$ and weight $\kappa$, there is a unique morphism $X \to \mathcal{C}_N$ lifting $\kappa$ such that $\cF_X$ is the pullback of $\mathcal{F}^{\mathrm{univ}}$.
  \end{quotation}
  
 \section{Two-variable $p$-adic $L$-functions}
  
  Let $U_1$ and $U_2$ be two affinoid subdomains of $\cW$. We write $\bfk_i: \Zp^\times \to \cO(U_i)^\times$ for the pullbacks of the canonical character $\bfk$. We suppose that we are given the following data:
  \begin{itemize}
   \item a finite flat covering $\tU_2 \to U_2$,
   \item  an overconvergent family $\cF_2 \in M_{\bfk_2}^\dagger(\Gamma_1(N); \tU_2)$ (not necessarily cuspidal or normalised),
   \item a locally analytic character $\tau \in \cW(L)$, with $t = w(\tau) \ge 0$.
  \end{itemize} 
  
  We define two families of nearly-overconvergent forms over $U_1 \times \tU_2$, both of weight $\bfk_1$ and degree of near-overconvergence $\le t$, by 
  \begin{align*}
   \Xi_\tau^{\ss} &\coloneqq \cF_2 \cdot \theta^{\tau}\left( E^{[p]}_{\bfk_1 - \bfk_2 - 2\tau}\right), \\
   \Xi_\tau^{\ds} &\coloneqq \cF_2 \cdot \theta^{\tau}\left( F^{[p]}_{\bfk_1 - \bfk_2 - 2\tau}\right).
  \end{align*}
  
  We apply to both of these forms the overconvergent projector $\Pi^{\mathrm{oc}}$ of \cite[\S 3.3.4]{Urban-nearly-overconvergent}. This gives elements 
  \[ 
   \Pi^{\mathrm{oc}}\left(\Xi_{\tau}^{\ss}\right),\ \Pi^{\mathrm{oc}}\left(\Xi_\tau^{\ds}\right) 
   \in \frac{1}{\prod_{m = 2}^{2t}\left( \nabla_1 - m \right)} S^{\dagger}_{\bfk_1}\left(\Gamma_1(N), U_1 \times \tU_2\right),
  \]
  where $\nabla_1 \in \cO(U_1)$ is the pullback to $U_1$ of the unique rigid-analytic function $\nabla \in \cO(W)$ such that $\nabla(\kappa) = w(\kappa)$ for all locally-algebraic $\kappa$.
  
  \begin{proposition}
   Let $(\kappa_1, \kappa_2)$ be a locally-algebraic point of $U_1 \times U_2$ such that $1 \le k_2 \le k_1 -1-t$, where $k_i = w(\kappa_i)$, and with $k_1 \notin \{2, \dots, 2t\}$. Let $\tilde\kappa_2$ be a point of $\tU_2$ above $\kappa_2$, and $f_2$ the specialisation of $\cF_2$ at $\tilde\kappa_2$. Let us suppose that $f_2$ is a classical modular form.
   
   Then the specialisations of $\Pi^{\mathrm{oc}}\left(\Xi_{\tau}^{\ss}\right)$ and $\Pi^{\mathrm{oc}}\left(\Xi_{\tau}^{\ds}\right)$ at $(\kappa_1, \tilde\kappa_2)$ are given by
   \begin{align*}
    \Pi^{\mathrm{oc}}\left(\Xi_{\tau}^{\ss}\right)(\kappa_1, \tilde\kappa_2) &= \Pi^{\mathrm{hol}}\Big( f_2 \cdot \cE_N(\kappa_1, \kappa_2, \kappa_1 - 1 - \tau)\Big),\\
    \Pi^{\mathrm{oc}}\left(\Xi_{\tau}^{\ds}\right)(\kappa_1, \tilde\kappa_2) &= \Pi^{\mathrm{hol}}\Big( f_2 \cdot \cE_N(\kappa_1, \kappa_2, \kappa_2 + \tau)\Big).
   \end{align*}
  \end{proposition}
  
  \begin{proof}
   An elementary computation shows that $\theta^{\tau}\left( E^{[p]}_{\kappa_1 - \kappa_2 - 2\tau}\right) = \cE_N(\kappa_1, \kappa_2, \kappa_1 - 1 - \tau)$ and similarly that $\theta^{\tau}\left( F^{[p]}_{\kappa_1 - \kappa_2 - 2\tau}\right) = \cE_N(\kappa_1, \kappa_2, \kappa_2 + \tau)$. The result now follows from the compatibility of the holomorphic and overconvergent projection operators.
  \end{proof}
  
  \begin{remark}
   We may consider the formal power series $\cF_2 \cdot \cE_N(\bfk_1, \bfk_2, \sigma)$ as a family of $p$-adic modular forms over $U_1 \times \tU_2 \times \cW$. This is not overconvergent, or even nearly-overconvergent, in any reasonable sense, since the near-overconvergence degrees of its specialisations are not bounded above over any open affinoid in the parameter space $U_1 \times \tU_2 \times \cW$. However, the above proposition gives two families of 2-dimensional ``slices'' of the parameter space for which the above family does become nearly-overconvergent, of bounded degree, over any given slice. 
  \end{remark}
     
  Let us now suppose that $k_1 \ge 2$ is a non-negative integer lying in $U_1$, $N_f$ is an integer dividing $N$, and $f_1 \in S_{k_1}(\Gamma_1(N_f) \cap \Gamma_0(p), L)$ is a ``noble eigenform'' in the sense of \cite[Definition 4.6.3]{loefflerzerbes16}; that is, $f_1$ is a $p$-stabilisation of some normalised newform of level $\Gamma_1(N_f)$ whose Hecke polynomial at $p$ has distinct roots, and a mild extra condition is satisfied in the case of critical-slope eigenforms.
  
  Then, after possibly shrinking the affinoid neighbourhood $U_1 \ni k_1$, we can find a Coleman family of normalised eigenforms $\cF_1$ over $U_1$ whose specialisation at $k_1$ is $f_1$; and a continuous $\cO(U_1)$-linear functional 
  \[ \lambda_{\cF_1^c}: S_{\bfk_1}^\dagger(\Gamma_1(N_f), U_1) \to \cO(U_1) \]
  factoring through the Hecke eigenspace associated to the dual family $\cF_1^c$, and mapping the normalised eigenform $\cF_1^c$ itself to 1. We extend this to a linear functional on forms of level $N$ by composing with the trace map. We can therefore define two meromorphic functions, both lying in the space $\tfrac{1}{\prod_{j = 2}^{2w(\tau)}\left( \nabla_1 - j \right)}\cO(U_1 \times \tU_2)$, by the formulae
  \[ 
   L_p^{\ss}(\cF_1, \cF_2; \tau) = N^{(-\bfk_1 + \bfk_2 + 2\tau)} \lambda_{\cF_1^c}\Big[ \Pi^{\mathrm{oc}}\left(\Xi_{\tau}^{\ss}\right)\Big], 
  \]
  and
  \[ 
   L_p^{\ds}(\cF_1, \cF_2; \tau) = N^{(\bfk_1 - \bfk_2 - 2\tau-2)}\lambda_{\cF_1^c}\Big[\Pi^{\mathrm{oc}}\left(\Xi_{\tau}^{\ds}\right)\Big].
  \]
  By construction, $L_p^{\ss}$ interpolates the values $I(f_1, f_2, \kappa_1 -1 -\tau)$, and $L_p^{\ds}$ the values $I(f_1, f_2, \kappa_2 + \tau)$, for varying $f_1$ and $f_2$ (but fixed $\tau$).
  
  \begin{remark}
   Our eventual goal is to show that there is a 3-variable $L$-function on $U_1 \times \tU_2 \times \cW$ interpolating all critical values of the Rankin $L$-function. The 2-variable $L$-functions $ L_p^{\ss}$ and $L_p^{\ds}$ will turn out to be slices of this 3-variable $L$-function, along two different families of 2-dimensional subspaces of the parameter space.
  \end{remark}
  
  Let us, finally, specialise to the case where $\tU_2$ is an affinoid subdomain of the eigencurve $\mathcal{C}_{N_2}$, and $\cF_2$ is the universal eigenform. One knows that $\mathcal{C}_{N_2}$ is admissibly covered by affinoids $\tU_2$ with the property that $\tU_2$ is a finite flat covering of an admissible open in $\cW$, as above; and the above construction is clearly compatible on overlaps, so we obtain two families of meromorphic functions on $U_1 \times \mathcal{C}_{N_2}$.

 \section{Compatibility of the two families}
  
  \begin{definition}
   Given a locally algebraic $\tau$ with $w(\tau) \ge 0$, we define two 2-dimensional rigid-analytic subspaces of $U_1 \times \tU_2 \times \cW$ by
   \[ 
    \cW^{\ss}(\tau) = \{ (\kappa_1, \tilde\kappa_2, \kappa_1 - 1 - \tau): \kappa_1 \in U_1, \tilde \kappa_2 \in U_2\} 
   \]
   and
   \[ 
    \cW^{\ds}(\tau) = \{ (\kappa_1, \tilde\kappa_2, \kappa_2 + \tau): \kappa_1 \in U_1, \tilde \kappa_2 \in U_2\}.
   \]
   
   We set $\Sigma_{\mathrm{crit}}^{\ss}(\tau) = \Sigma_{\mathrm{crit}} \cap \cW^{\ss}(\tau)$ and similarly $\Sigma_{\mathrm{geom}}^{\ss}(\tau)$, $\Sigma_{\mathrm{crit}}^{\ds}(\tau)$, $\Sigma_{\mathrm{geom}}^{\ds}(\tau)$.
  \end{definition}
  
  We can then regard $L_p^{\ss}(\cF_1, \cF_2; \tau)$ as a $p$-adic meromorphic function on $\cW^\ss(\tau)$ in a natural way, interpolating classical $L$-values at the points in $\Sigma_{\mathrm{crit}}^{\ss}(\tau)$; and similarly for $\ds$.
  
  We have the following technical lemma:
  
  \begin{lemma}
   Let $\tau, \tau'$ be two locally-algebraic characters with $w(\tau) \ge 0, w(\tau') \ge 0$, and suppose that we have
   \[ \{ \kappa - (1 + \tau + \tau') : \kappa \in U_1\} \subseteq U_2.\]
   Then $L_p^{\ss}(\cF_1, \cF_2; \tau)$ and $L_p^{\ds}(\cF_1, \cF_2; \tau')$ coincide as functions on $\cW^\ss(\tau) \cap \cW^\ds(\tau')$.
  \end{lemma}
  
  \begin{proof}
   The intersection $\cW^\ss(\tau) \cap \cW^\ds(\tau')$ consists of those points of the form $(\kappa_1, \tilde \kappa_2, \kappa_1 - 1-\tau)$ such that $\tilde\kappa_2$ lies above the point $\kappa_1 - (1 + \tau + \tau')$ of $\cW$. In particular, under the assumptions of the lemma, this is simply a finite covering of $U_1$.
   
   Let $(\kappa_1, \tilde\kappa_2, \sigma)$ be a point in this intersection with $\kappa_1$ locally algebraic, and such that $w(\kappa_1) \ge 2\max(w(\tau), w(\tau')) + 1$ in order to avoid singularities of the nearly-overconvergent projection operators. Then the two $p$-adic $L$-functions specialise to the image under $\lambda_{f_1^c}$ of the nearly-overconvergent modular forms with $q$-expansions
   \[ f_2 \theta^{\tau}\left( E^{[p]}_{\kappa_1 - \kappa_2 - 2\tau} \right) \quad\text{and}\quad
   f_2 \theta^{\tau'}\left( F^{[p]}_{\kappa_1 - \kappa_2 - 2\tau'} \right).\]
   Since these two modular forms are identical, we deduce that the two $L$-functions agree at the given point. As the set of locally-algebraic $\kappa_1 \in U_1$ with $w(\kappa_1)$ greater than any given bound is clearly Zariski-dense, it follows that the two $p$-adic $L$-functions are identically equal on this intersection.
  \end{proof}
  
  \begin{lemma}
   \label{lemma:intersect}
   Let $\tau$ be a locally algebraic character with $w(\tau) \ge 0$. If $U_2$ is sufficiently large (depending on $U_1$ and $\tau$), then the union of the intersections $\cW^{\ss}(t) \cap \cW^\ds(\tau)$, as $t$ varies over integers $\ge 0$, is Zariski dense in $\cW^\ds(\tau)$.
  \end{lemma}
  
  \begin{proof}
   Easy check.
  \end{proof}
 
 \section{The 3-variable geometric $L$-function}
  
  We now turn from ``$p$-adic analytic'' methods to ``arithmetic'' ones -- that is, we invoke the existence of the Euler system of Beilinson--Flach elements.
  
  \begin{theorem}
   Suppose $\widetilde U_2$ is the preimage of $U_2$ in the \emph{ordinary} locus of the eigencurve, and $\cF_2$ the universal ordinary family over $U_2$. Then there exists a $p$-adic meromorphic\footnote{It is analytic if the product of the prime-to-$p$ nebentypus characters of $\cF_1$ and $\cF_2$ is non-trivial. Otherwise, it may have poles along the near-central points $(\kappa_1, \kappa_2, \sigma)$ such that $\kappa_1 + \kappa_2 = 2\sigma$. This is a consequence of the `smoothing factors' $c^2 - c^?$ appearing in the construction of the Beilinson--Flach elements. In particular, the restriction of $L_p^{\mathrm{geom}}$ to any $\ds$ or $\ss$ slice is well-defined.} function $L_p^{\mathrm{geom}}(\cF_1, \cF_2)$ on $U_1 \times \tU_2 \times \cW$ with the following property:
   
   \begin{enumerate}
    \item[\textup{($\dagger$)}] For any crystalline character $\tau = t$ with $t \ge 0$, the 2-variable $p$-adic $L$-function $L_p^{\ss}(\cF_1, \cF_2; \tau)$ is the restriction of $L_p^{\mathrm{geom}}$ to $\cW^{\ss}(\tau)$.
   \end{enumerate}
   
   Moreover, $L_p^{\mathrm{geom}}$ is related to the Euler system of Beilinson--Flach elements via the formula
   \[ L_p^{\mathrm{geom}}(\cF_1, \cF_2) = \Big(c^2 - \varepsilon_{N, 1}(c)^{-1} \varepsilon_{N, 2}(c)^{-1} c^{2\mathbf{s}+2-\bfk_1 - \bfk_2}\Big)^{-1} (-1)^{\mathbf{s}} 
   \lambda(\cF_1)^{-1}\left\langle {}_c \mathcal{BF}^{[\cF_1, \cF_2]}, \eta_{\cF_1} \otimes \omega_{\cF_2} \right\rangle\]
   in the notation of \cite[\S 9.1]{loefflerzerbes16}, for any $c > 1$ coprime to $6pN_1 N_2$.
  \end{theorem}
  
  \begin{proof}
   This is essentially proved in \cite[\S 9.3]{loefflerzerbes16}. The only difference in our present statement is that we are allowing $U_2$ to be arbitrary, and permitting some finite flat covering $\tU_2 \to U_2$, whereas in our earlier work we assumed both $U_1$ and $U_2$ were small neighbourhoods of some given eigenforms $f_1, f_2$. However, the latitude to shrink $U_2$ was only used in \emph{op.cit.} at precisely two points: 
   \begin{itemize}
    \item in the proof of Proposition 5.3.4 of \emph{op.cit.}, in order to arrange that all specialisations of $\cF_2$ at points of classical weight were classical; this is automatically satisfied for ordinary families.
    \item in Sections 6.3 and 6.4 of \emph{op.cit.}, in order to find a triangulation of the $(\varphi, \Gamma)$-module associated to $\cF_2$, and canonical crystalline periods for the filtration steps; this can be carried out globally over an ordinary family, using Ohta's results \cite{ohta00}, as in \cite{KLZ17}.\qedhere
   \end{itemize}
  \end{proof}
  
  In order to complete the proof, we shall manoeuvre from the rather weak interpolating property $(\dagger)$ of $L_p^{\mathrm{geom}}$ into a much stronger one, by repeatedly using the compatibility between the $\ss$ and $\ds$ slices.
  
  \begin{corollary}
   Let $\tau$ be any locally-algebraic character (not necessarily crystalline) with $w(\tau) \ge 0$. If $U_2$ is sufficiently large (depending on $U_1$ and $\tau$) then
   \[ L_p^\ds(\cF_1, \cF_2; \tau) = L_p^{\mathrm{geom}}(\cF_1, \cF_2) |_{\cW^{\ds}(\tau)} \]
   and
   \[  L_p^\ss(\cF_1, \cF_2; \tau) = L_p^{\mathrm{geom}}(\cF_1, \cF_2) |_{\cW^{\ss}(\tau)}.\]
  \end{corollary}
  
  \begin{proof}
   By Lemma \ref{lemma:intersect}, for the first equality, it suffices to show that $L_p^\ds$ and $L_p^{\mathrm{geom}}$ agree on the intersection $\cW^{\ss}(t) \cap \cW^{\ds}(\tau)$, for integers $t \ge 0$. However, we know that $L_p^\ds$ and $L_p^\ss$ coincide on these intersections, and that $L_p^{\mathrm{geom}}$ in turn coincides with $L_p^{\ss}$.
   
   For the second equality, we consider the intersection of $\cW^{\ss}(\tau)$ with the slices $\cW^{\ds}(\tau')$, where $\tau'$ is an arbitrary locally-algebraic character of weight $w(\tau') \ge 0$. Using the previously-proved equality, we know that $L_p^{\mathrm{geom}}$ agrees with $L_p^\ss(\cF_1, \cF_2; \tau)$ on each of these intersections. As before, the union of these is Zariski dense in $\cW^{\ss}(\tau)$ as required.
  \end{proof}
  
  We conclude, finally, the following interpolation formula. Recall that we are assuming $\cF_2$ to be an ordinary family.
  
  \begin{theorem}
   Let $(\kappa_1, \tilde\kappa_2, \sigma)$ be a triple of locally-algebraic points in $U_1 \times \tU_2 \times \cW$, with $1 \le w(\kappa_2) \le w(\sigma) \le w(\kappa_1) - 1$. Let $f_1$, $f_2$ be the specialisations of $\cF_1, \cF_2$ at the weights $\kappa_i$, and suppose that these specialisations are classical.
   
   Then we have
   \[ L_p^{\mathrm{geom}}(\cF_1, \cF_2)(\kappa_1, \tilde\kappa_2, \sigma) = I(f_1, f_2, \sigma).\]
  \end{theorem}
  
  \begin{proof}
   Given any such triple, let us write $\tau = \kappa_1 - 1 - \sigma$ and $\tau' = \sigma - \kappa_2$. Both of these are locally algebraic characters, and $w(\tau), w(\tau') \ge 0$.
   
   Since $w(\tau) + w(\tau') = w(\kappa_1) -1 -w(\kappa_2)$, at least one of the quantities $w(\tau)$ and $w(\tau')$ must be $\le \tfrac{w(\kappa_1) - 1}{2}$. If $w(\tau) \le \tfrac{w(\kappa_1) - 1}{2}$, then  $(\kappa_1, \tilde\kappa_2, \sigma)$ lies in the interval in which $L_p^{\spadesuit}(\cF_1, \cF_2; \tau)$ interpolates the classical Rankin--Selberg period. Similarly, if $w(\tau')$ is smaller than this bound we may invoke the interpolating property of $L_p^{\ds}$.
   
   Since $\cF_2$ is an ordinary family, we may assume without loss of generality that $U_2$ is arbitrarily large, and via the previous theorem, we can conclude that $L_p^{\ss}$ or $L_p^{\ds}$ coincides with the appropriate specialisation of the 3-variable $p$-adic $L$-function. 
  \end{proof}
  
 \appendix
 
 \section{Evaluation of the Rankin--Selberg period}
 
  For the convenience of the reader, we outline the derivation of the formula relating the period $I(f_1, f_2, \sigma)$ defined above to the Rankin--Selberg $L$-function. Our approach is closely based on that of \cite{perrinriou88}. We place ourselves in the setting of Proposition \ref{prop:interp-formula}; and, since the case of trivial $\chi$ is covered in many references, we shall assume that $\chi$ is non-trivial, of conductor $p^r$ with $r \ge 1$.
  
  \subsubsection*{Step 1}
   
   We express the linear functional $\lambda_{f_1^c}$ on $S_k\left(\Gamma_1(N) \cap \Gamma_0(p^n)\right)$, for any $n \ge 1$, via the formula
   \[
    \lambda_{f_1^c}(h) = 
     \left( \frac{\varepsilon_1(p)}{\alpha_1}\right)^{n-1} \cdot
     \frac{ 
      \left \langle g_n
      , h \right\rangle_{N(p^{n})}
     }
     { \left\langle g, f_1^c \right\rangle_{N_1(p)}},
   \]
   where $g = W_{N_1 p}(f_{1,\beta})$ and $g_n = g \mid_k \left(\begin{smallmatrix} p^{n-1} \\ & 1 \end{smallmatrix}\right)$. Here $f_{1, \beta}$ is the $p$-stabilisation of $f_1^\circ$ corresponding to the root $\beta_1$ of the Hecke polynomial; and the subscript $N(p^n)$ denotes the Petersson product at level $\Gamma_1(N) \cap \Gamma_0(Np^n)$. Cf.~\cite[Proposition 4.5]{hida85}. A computation closely analogous to the final step of \cite[Proposition 10.1.1]{KLZ17} shows that the denominator term is given by
   \[ 
    \left\langle g, f_1^c \right\rangle_{N_1(p)} = \frac{\overline{\lambda(f_1^\circ)} \alpha \cE(f_1)  \cE^*(f_1)}{\varepsilon_1(p)}\cdot \langle f_1^\circ, f_1^\circ \rangle_{N_1},
   \]
   where $\lambda(f_1^\circ)$ denotes the Atkin--Lehner pseudo-eigenvalue of $f_1^\circ$. This yields the formula
   \[ I(f_1, f_2, j + \chi) = 
   \frac{\varepsilon_1(p)^{2r}}{\alpha_1^{2r}  \overline{\lambda(f_1^\circ)}  \cE(f_1)  \cE^*(f_1) \langle f_1^\circ, f_1^\circ \rangle_{N_1}} \left\langle g_n, f_2 \cdot \cE(k_1, k_2, j + \chi) \right\rangle_{N(p^{2r})}.\]
   
  \subsubsection*{Step 2}
   
   We recognise the nearly-holomorphic Eisenstein series $\cE(k_1, k_2, j + \chi)$ of level $Np^{2r}$ as the twist by the character $\chi$ of a simpler Eisenstein series $\tilde E$ of level $Np^r$ and character $\chi^{-2}$, whose $q$-expansion is
   \[ 
    \sum_{n \ge 1} q^n \sum_{\substack{d \mid n \\ p \nmid \tfrac{n}{d}}} d^{j-k_2} (n/d)^{k_1 - 1 - j} \chi(n/d)^{-2} \left(e^{2\pi i d / N} + (-1)^{k_1 - k_2} e^{-2\pi i d / N}\right).
   \]
   Since $a_n(g_{2r}) = 0$ unless $p^{2r-1} \mid n$, we can pull the twist through the Petersson product to write
   \[ 
    \left\langle g_{2r}, f_2 \cdot \cE(k_1, k_2, j + \chi) \right\rangle_{N(p^{2r})} = \chi(-1) 
    \left \langle g_{2r}, f_{2, \chi} \cdot \tilde E \right\rangle_{N(p^{2r})}.
   \]
  
  \subsubsection*{Step 3}
  
   We re-write the last Petersson product using the local Atkin--Lehner operator $W_{p^{2r}}$ acting on forms of level $Np^{2r}$. We compute that
   \[ 
    \tilde E \mid W_{p^{2r}} = p^{2r(k_1-2-j)}\chi(-1) \sum_{a \in (\ZZ/p^{2r}\ZZ)^\times} \chi(a)^{-2} E_{1/N + a/p^{2r}}
   \]
   where the nearly-holomorphic Eisenstein series $E_\gamma = E_{\gamma}^{k_1 - k_2}(-, j-k_1 + 1)$ for $\gamma \in \QQ/\ZZ$ is as in \cite[\S 4--5]{leiloefflerzerbes14}. On the other hand, the action on $f_{2, \chi}$ is given by
   \[ f_{2, \chi} \mid W_{p^{2r}} = p^{(k_2-3)r} \varepsilon_2(p)^r G(\chi)^2 f_{2, \chi^{-1}}. \]
   Combining these formulae we deduce
   \[ 
    \left\langle g_{2r}, f_2 \cdot \cE(k_1, k_2, j + \chi) \right\rangle_{N(p^{2r})}
    = 
    \left( \frac{p^{(2k_1 + k_2 -5 -2j)r}  G(\chi)^2 \chi(N^2)}{\varepsilon_1(p)^{2r} \varepsilon_2(p)^{r}}\right) \left\langle f_{1, \beta} \mid_{k_1} W_{N_1}, f_{2, \chi^{-1}} \cdot  E_{1/Np^{2r}}\right\rangle_{Np^{2r}}.
   \]
   
   \subsubsection*{Step 4} 
   
    Via the classical ``unfolding'' technique, integrating against the Eisenstein series $E_{1/Np^{2r}}$ gives the (imprimitive) Rankin--Selberg $L$-function at $s = j$; cf.~\cite[Theorem 7.1]{kato04}. That is, we have
    \[ 
     \left\langle f_{1, \beta} \mid_{k_1} W_{N_1}, f_{2, \chi^{-1}} \cdot  E_{1/Np^{2r}}\right\rangle_{Np^{2r}}
     = 
     \frac{(j-1)! (j-k_2)! i^{k_1-k_2} L^{\mathrm{imp}}\left(\overline{f_{1, \beta} \mid_k W_{N_1}}, f_{2, \chi^{-1}}, j\right)}
     {N^{k_1 + k_2 -2j - 2} p^{2r(k_1 + k_2-2j-2)} \pi^{2j+1-k_2} 2^{2j + k_1 - k_2}}
     .
    \]
    However, since all Fourier coefficients $a_n$ of $f_{2, \chi^{-1}}$ with $p \mid n$ are zero, this formula is unchanged if we replace $\overline{f_{1, \beta} \mid_k W_{N_1}}$ with any form having the same Fourier coefficients away from $p$; one such form is $\overline{\lambda(f_1^\circ)} f_1^\circ$, so this is
    \[ 
     \left\langle f_{1, \beta} \mid_{k_1} W_{N_1}, f_{2, \chi^{-1}} \cdot  E_{1/Np^{2r}}\right\rangle_{Np^{2r}}
     = 
     \frac{(j-1)! (j-k_2)! i^{k_1-k_2} \overline{\lambda(f_1^\circ)}\cdot L^{\mathrm{imp}}\left(f_1^\circ, f_2^\circ, \chi^{-1}, j\right)}
     {N^{k_1 + k_2 -2j - 2} p^{2r(k_1 + k_2-2j-2)} \pi^{2j+1-k_2} 2^{2j + k_1 - k_2}}.
    \]
    Combining steps 1, 3 and 4 gives the formula stated in Proposition \ref{prop:interp-formula}. A similar argument (using an Eisenstein series of level $Np^{r + r'}$) can be used to prove Proposition \ref{prop:interp-formula2}.

\providecommand{\bysame}{\leavevmode\hbox to3em{\hrulefill}\thinspace}
\renewcommand{\MR}[1]{%
 MR \href{http://www.ams.org/mathscinet-getitem?mr=#1}{#1}.
}
\newcommand{\articlehref}[2]{\href{#1}{#2}}


\begin{thebibliography}{BDR15b}

\bibitem[AIU]{AIU}
Fabrizio Andreatta and Adrian Iovita (with an appendix by Eric Urban). \emph{Triple product $p$-adic $L$-functions associated to finite slope $p$-adic families of modular forms.} In preparation.

\bibitem[BDR15a]{BDR-BeilinsonFlach}
Massimo Bertolini, Henri Darmon, and Victor Rotger,
  \articlehref{http://dx.doi.org/10.1090/S1056-3911-2014-00670-6}{\emph{{B}eilinson--{F}lach
  elements and {E}uler systems {I}: syntomic regulators and {$p$}-adic {R}ankin
  {$L$}-series}}, J. Algebraic Geom. \textbf{24} (2015), no.~2, 355--378.
  \MR{3311587}

\bibitem[BDR15b]{BDR-BeilinsonFlach2}
\bysame,
  \articlehref{http://dx.doi.org/10.1090/S1056-3911-2015-00675-0}{\emph{{B}eilinson--{F}lach
  elements and {E}uler systems {II}: the {B}irch and {S}winnerton-{D}yer
  conjecture for {H}asse--{W}eil--{A}rtin {$L$}-functions}}, J. Algebraic Geom.
  \textbf{24} (2015), no.~3, 569--604. \MR{3344765}

\bibitem[BL16a]{buyukboduklei16}
Kazim B\"uy\"ukboduk and Antonio Lei,
  \articlehref{http://arxiv.org/abs/1602.07508}{\emph{{A}nticyclotomic
  {$p$}-ordinary {I}wasawa theory of elliptic modular forms}}, 2016,
  \path{arXiv:1602.07508}.

\bibitem[BL16b]{buyukboduklei16b}
\bysame,
  \articlehref{http://arxiv.org/abs/1605.05310}{\emph{{A}nticyclotomic
  {I}wasawa theory of elliptic modular forms at non-ordinary primes}}, 2016,
  \path{arXiv:1605.05310}.

\bibitem[Cas15]{castella-heights-BF}
Francesc Castella, \articlehref{http://arxiv.org/abs/1509.02761}{\emph{P-adic
  heights of {H}eegner points and {B}eilinson--{F}lach elements}}, preprint,
  2015, \path{arXiv:1509.02761}.

\bibitem[CE98]{colemanedixhoven98}
Robert Coleman and Bas Edixhoven,
  \articlehref{http://dx.doi.org/10.1007/s002080050140}{\emph{On the
  semi-simplicity of the {$U_p$}-operator on modular forms}}, Math. Ann.
  \textbf{310} (1998), no.~1, 119--127. \MR{1600034}

\bibitem[Das16]{Dasgupta-factorization}
Samit Dasgupta,
  \articlehref{http://dx.doi.org/10.1007/s00222-015-0634-4}{\emph{Factorization
  of {$p$}-adic {R}ankin {$L$}-series}}, Invent. Math. \textbf{205} (2016),
  no.~1, 221--268.

\bibitem[Hid85]{hida85}
Haruzo Hida, \articlehref{http://dx.doi.org/10.1007/BF01388661}{\emph{A
  {$p$}-adic measure attached to the zeta functions associated with two
  elliptic modular forms. {I}}}, Invent. Math. \textbf{79} (1985), no.~1,
  159--195. \MR{774534}

\bibitem[Hid88]{hida88}
\bysame, \articlehref{http://www.numdam.org/item?id=AIF_1988__38_3_1_0}{\emph{A
  {$p$}-adic measure attached to the zeta functions associated with two
  elliptic modular forms. {II}}}, Ann. Inst. Fourier (Grenoble) \textbf{38}
  (1988), no.~3, 1--83. \MR{976685}

\bibitem[JSW15]{jetchevskinnerwan}
Dimitar Jetchev, Chris Skinner, and Xin Wan,
  \articlehref{http://arxiv.org/abs/1512.06894}{\emph{The
  {B}irch--{S}winnerton-{D}yer formula for elliptic curves of analytic rank
  one}}, preprint, 2015, \path{arXiv:1512.06894}.

\bibitem[Kat04]{kato04}
Kazuya Kato,
  \articlehref{http://smf4.emath.fr/en/Publications/Asterisque/2004/295/html/smf_ast_295_117-290.html}{\emph{{$P$}-adic
  {H}odge theory and values of zeta functions of modular forms}},
  Ast{\'e}risque \textbf{295} (2004), ix, 117--290, Cohomologies $p$-adiques et
  applications arithm{\'e}tiques. III. \MR{2104361}

\bibitem[KLZ17]{KLZ17}
Guido Kings, David Loeffler, and Sarah~Livia Zerbes,
  \articlehref{http://dx.doi.org/10.4310/CJM.2017.v5.n1.a1}{\emph{{R}ankin--{E}isenstein
  classes and explicit reciprocity laws}}, Cambridge J. Math. \textbf{5}
  (2017), no.~1, 1--122.

\bibitem[LLZ14]{leiloefflerzerbes14}
Antonio Lei, David Loeffler, and Sarah~Livia Zerbes,
  \articlehref{http://dx.doi.org/10.4007/annals.2014.180.2.6}{\emph{Euler
  systems for {R}ankin--{S}elberg convolutions of modular forms}}, Ann. of
  Math. (2) \textbf{180} (2014), no.~2, 653--771. \MR{3224721}

\bibitem[LZ16]{loefflerzerbes16}
David Loeffler and Sarah~Livia Zerbes,
  \articlehref{http://dx.doi.org/10.1186/s40687-016-0077-6}{\emph{Rankin--{E}isenstein
  classes in {C}oleman families}}, Res. Math. Sci. \textbf{3} (2016), no.~29,
  special collection in honour of Robert F.\ Coleman.

\bibitem[My91]{my91}
Vinh~Quang My, \emph{Rankin non-{A}rchimedean convolutions of unbounded
  growth}, Mat. Sb. (N.S.) \textbf{182} (1991), no.~2, 164--174, translation in
  Math. USSR-Sb. 72 (1992), no. 1, 151--161.

\bibitem[Oht00]{ohta00}
Masami Ohta,
  \articlehref{http://dx.doi.org/10.1007/s002080000119}{\emph{Ordinary
  {$p$}-adic \'etale cohomology groups attached to towers of elliptic modular
  curves. {II}}}, Math. Ann. \textbf{318} (2000), no.~3, 557--583. \MR{1800769}

\bibitem[Pan82]{panchishkin82}
Alexei Panchishkin, \emph{Le prolongement {$p$}-adique analytique des fonctions
  {$L$} de {R}ankin. {I}}, C. R. Acad. Sci. Paris S{\'e}r. I Math. \textbf{295}
  (1982), no.~2, 51--53.

\bibitem[PR88]{perrinriou88}
Bernadette Perrin-Riou,
  \articlehref{http://dx.doi.org/10.1112/jlms/s2-38.1.1}{\emph{Fonctions {$L$}
  {$p$}-adiques associ{\'e}es {\`a} une forme modulaire et {\`a} un corps
  quadratique imaginaire}}, J. London Math. Soc. \textbf{38} (1988), no.~1,
  1--32. \MR{949078}

\bibitem[Urb14]{Urban-nearly-overconvergent}
Eric Urban,
  \articlehref{http://dx.doi.org/10.1007/978-3-642-55245-8_14}{\emph{Nearly
  overconvergent modular forms}}, Iwasawa Theory 2012: State of the Art and
  Recent Advances (Berlin) (Thanasis Bouganis and Otmar Venjakob, eds.),
  Contributions in Mathematical and Computational Sciences, vol.~7, Springer,
  2014, pp.~401--441.

\bibitem[Wan15]{wan15}
Xin Wan, \articlehref{http://arxiv.org/abs/1411.6352v2}{\emph{Iwasawa main
  conjecture for supersingular elliptic curves}}, preprint, 2015,
  \path{arXiv:1411.6352v2}.

\end{thebibliography}
\end{document}